\documentclass{amsart}
\usepackage{amsfonts,amscd,amsthm,amsgen,amsmath,amssymb}
\usepackage[all]{xy}
\usepackage{epsfig,color}
\usepackage[vcentermath]{youngtab}
\newtheorem{theorem}{Theorem}[section]
\newtheorem{lemma}[theorem]{Lemma}
\newtheorem{corollary}[theorem]{Corollary}
\newtheorem{proposition}[theorem]{Proposition}

\theoremstyle{definition}
\newtheorem{definition}[theorem]{Definition}

\theoremstyle{remark}
\newtheorem{remark}[theorem]{Remark}

\numberwithin{equation}{section}

\begin{document}

\title{The alpha invariant of complete intersections}
\author{David Baraglia}




\date{\today}

\begin{abstract}

We compute the alpha invariant of any smooth complex projective spin complete intersection of complex dimension $1 \; ({\rm mod} \; 4)$. We prove that the alpha invariant depends only on the total degree and Pontryagin classes. Our findings are consistent with a long-standing conjecture, often called the Sullivan Conjecture, which states that two complete intersections with the same dimensions, total degrees, Pontryagin and Euler classes are diffeomorphic.

\end{abstract}

\maketitle


\section{Introduction}

A smooth complex projective complete intersection $X_n(d_1, \dots , d_k)$ is a subvariety of $n+k$-dimensional complex projective space cut out by the transversal intersection of $k$ non-singular hypersurfaces of degrees $d_1, \dots , d_k$. In particular $X_n(d_1, \dots , d_k)$ is a non-singular complex projective variety of dimension $n$. By a classical result of Thom the diffeomorphism type of the underlying smooth $2n$-manifold depends only on $n , d_1, \dots , d_k$. We write $\underline{d} = (d_1, \dots , d_k)$ and abbreviate $X_n(d_1, \dots , d_k)$ to $X_n(\underline{d})$. We call $\underline{d}$ the {\em multi-degree} of $X_n(\underline{d})$.

From the Lefschetz hyperplane theorem, $H^2( X_n(\underline{d}) ; \mathbb{Z}) \cong \mathbb{Z}$ whenever $n \ge 3$. As will be recalled in Section \ref{sec:pont}, the $j$-th Pontryagin class must be an integral multiple of $x^{2j}$ where $x$ is a generator of $H^2( X_n(\underline{d}) ; \mathbb{Z})$. Therefore we can compare the Pontryagin classes of different complete intersections of the same dimension. We also have that $x^n$ is $d_{tot}$ times a generator of $H^{2n}( X_n(\underline{d}) ; \mathbb{Z})$, where $d_{tot} = d_1 d_2 \cdots d_k$ is called the {\em total degree} of $X_n(\underline{d})$. Hence $d_{tot}$ is a diffeomorphism invariant of $X_n(\underline{d})$ whenever $n \ge 3$.

In general, complete intersections of different multi-degrees may give rise to diffeomorphic manifolds, the simplest example being $X_1(1) = X_2(2) = S^2$. An outstanding problem in the study of complete intersections then is to determine precisely the conditions under which two complete intersections $X_n(\underline{d}), X_n(\underline{d}')$ can be diffeomorphic. Assuming $n \ge 3$ we clearly need $X_n(\underline{d}), X_n(\underline{d}')$ to have the same total degree, Pontryagin classes and Euler class. The following conjecture, sometimes called the Sullivan Conjecture, says that these conditions are also sufficient:\\

\noindent{\bf The Sullivan Conjecture for complete intersections:} Let $X, X'$ be $n$-dimensional complete intersections for $n\ge 3$ with the same total degrees, same Pontryagin classes and same Euler classes. Then $X$ and $X'$ are diffeomorphic.\\

For $n=3$ the Sullivan Conjecture holds as a consequence of the classification of simply-connected $6$-manifolds \cite{wall,jupp}. For $4 \le n \le 7$ the Sullivan Conjecture is known to hold up to homeomorphism \cite{fakl,fw}. The conjecture is also known to hold in arbitrary dimension under a certain condition on the total degree \cite{tra,kre}. Recent progress by Crowley and Nagy proves the smooth version for $n=4$ \cite{crna}. Crowley and Nagy further expect to prove the $n=5$ case in future work \cite{crna2}.

Looking to higher dimensions, one might try to produce a counter-example to the Sullivan Conjecture by finding a diffeomorphism invariant which distinguishes a pair of complete intersections $X_n(\underline{d}),X(\underline{d}')$ with the same total degree, Pontryagin and Euler classes. For $n \ge 2$, $X_n(\underline{d})$ is simply-connected hence if $X_n(\underline{d})$ is spin then it has a unique spin structure and we may consider the refined $\hat{A}$-genus or $\alpha$ invariant of $X_n(\underline{d})$:
\[
\alpha(X_n(\underline{d})) \in KO^{-2n}(pt) = \begin{cases} \mathbb{Z} & \text{if } n = 0,2 \; ({\rm mod} \; 4), \\ \mathbb{Z}_2 & \text{if } n = 1 \; ({\rm mod} \; 4), \\ 0 & \text{if } n = 3 \; ({\rm mod} \; 4). \end{cases}
\]
In the case $n$ is even, the $\alpha$ invariant is determined by the $\hat{A}$-genus $\hat{A}[ X_n(\underline{d})]$ which is easily seen to depend only on the Pontryagin classes and total degree. Hence we do not get a new diffeomorphism invariant in this case. On the other hand if $n = 1 \; ({\rm mod} \; 4)$ then $\alpha$ invariant is the mod $2$ index, which a priori is not determined by the Pontryagin classes. The main result of this paper is a proof that the $\alpha$ invariant is in fact determined by the Pontryagin classes and total degree, consistent with the Sullivan Conjecture:

\begin{theorem}\label{thm:main}
Let $X = X_n(\underline{d}), X' = X_n(\underline{d}')$, where $\underline{d} = (d_1 , \dots , d_k)$, $\underline{d}' = (d'_1 , \dots , d'_{k'})$ be two spin complete intersections of dimension $n \ge 5$, $n = 1 \; ({\rm mod} \; 4)$. Suppose that $X,X'$ have the same Pontryagin classes and the same total degree. Then $\alpha(X) = \alpha(X')$. Hence the alpha invariant of a spin complete intersection $X$ depends only on the dimension, total degree and Pontryagin classes of $X$.
\end{theorem}

Theorem \ref{thm:main} is thus a further piece of evidence in favour of the Sullivan Conjecture. Note that Theorem \ref{thm:main} does not actually tell us how to compute $\alpha$ in terms of the total degree and Pontryagin classes. In the remainder of the paper we prove a number of formulas which can be used to compute the invariant. Let $n = 1 \; ({\rm mod} \; 4)$, and let $\underline{d} = (d_1 , \dots , d_k)$ be such that an even number of $d_i$ are even. This determines a distinguished spin structure on $X_n(\underline{d})$ which for $n > 1$ is the unique spin structure. Let $\alpha_n(\underline{d}) \in \mathbb{Z}_2$ denote the alpha invariant of $X_n(\underline{d})$.

\begin{theorem}
The alpha invariant $\alpha_n(\underline{d})$ is given by any of the following descriptions:
\begin{itemize}
\item[(i)]{
\[
\alpha_n(\underline{d}) = \sum_{\epsilon_1 , \epsilon_2 , \dots , \epsilon_k = \pm 1} \binom{\frac{1}{2}(n+k+1+\underline{\epsilon} \cdot \underline{d}) }{n+k+1} \quad ({\rm mod} \; 2)
\]
where $\underline{\epsilon} = (\epsilon_1 , \epsilon_2 , \dots , \epsilon_k)$ and $\underline{\epsilon} \cdot \underline{d} = \epsilon_1 d_1 + \epsilon_2 d_2 + \cdots + \epsilon_k d_k$.
}
\item[(ii)]{
\[
\alpha_n(\underline{d}) = \sum_{j_1 + \cdots + j_k = n+1 \atop j_i+d_i \text{ odd}} \binom{\frac{1}{2}(d_1+j_1-1)}{j_1} \cdots \binom{\frac{1}{2}(d_k+j_k-1)}{j_k} \quad ({\rm mod} \; 2).
\]
}
\item[(iii)]{
$\alpha_n(\underline{d})$ is the mod $2$ reduction of the coefficient of $t^m$ in
\[
\left( \frac{1}{1-t} \right)^{n+k+1} (1-t^{d_1})(1-t^{d_2}) \cdots (1 - t^{d_k}),
\]
where $2m = -n-k-1+\sum_i d_i$.
}
\end{itemize}
\end{theorem}

With some modifications, the same techniques can be used to compute the $\hat{A}$-genus of a spin complete intersection $X_n(\underline{d})$ when $n$ is even. Let $n \ge 2$ be even and let $\underline{d} = (d_1 , \dots , d_k)$ be such that an odd number of $d_i$ are even. Then $X_n(\underline{d})$ has a unique spin structure. Let $\hat{A}[X_n(\underline{d})] \in \mathbb{Z}$ denote the $\hat{A}$-hat genus of $X_n(\underline{d})$.

\begin{theorem}
The $\hat{A}$-genus of $X_n(\underline{d})$ is given by either of the following descriptions:
\begin{itemize}
\item[(i)]{
\[
\hat{A}[X_n(\underline{d})] = \sum_{\epsilon_1 , \epsilon_2 , \dots , \epsilon_k = \pm 1} sgn( \underline{\epsilon} ) \binom{\frac{1}{2}(n+k-1+\underline{\epsilon} \cdot \underline{d}) }{n+k}
\]
where $\underline{\epsilon} = (\epsilon_1 , \epsilon_2 , \dots , \epsilon_k)$, $sgn(\underline{\epsilon}) = \epsilon_1 \cdots \epsilon_k$ and $\underline{\epsilon} \cdot \underline{d} = \epsilon_1 d_1 + \epsilon_2 d_2 + \cdots + \epsilon_k d_k$.
}
\item[(ii)]{
$\hat{A}[X_n(\underline{d})]$ is twice the coefficient of $t^m$ in
\[
\left( \frac{1}{1-t} \right)^{n+k+1} (1-t^{d_1})(1-t^{d_2}) \cdots (1 - t^{d_k}),
\]
where $2m = -n-k-1+\sum_i d_i$.
}
\end{itemize}
\end{theorem}

A brief outline of the contents of this paper is as follows. In Section \ref{sec:ci} we introduce the alpha invariant of spin complete intersections and give a $KO$-theoretic expression for it. Section \ref{sec:pont} is a short section on the Pontryagin classes of complete intersections. In Section \ref{sec:main} we prove the main result of the paper that the alpha invariant depends only on the dimension, total degree and Pontryagin classes. In Section \ref{sec:formulas} we give several expressions that can be used to compute the alpha invariant. Section \ref{sec:n=1} deals with the special case of $1$-dimensional complete intersections and lastly in Section \ref{sec:ahat} we give some analogous expressions for the $\hat{A}$-genus of a spin complete intersection of even dimension.\\

\noindent{\bf Acknowledgments.} The author wishes to thank Diarmuid Crowley and Csaba Nagy for introducing him to the problem of computing the alpha invariant of complete intersections and explaining its relevance to the Sullivan Conjecture. The author was financially supported by the Australian Research Council Discovery Project DP170101054.


\section{Complete intersections and the alpha invariant}\label{sec:ci}

Let $n,d_1, \dots , d_k \ge 1$ be positive integers. As in the introduction $X_n(d_1,\dots, d_k) = X_n(\underline{d})$ denotes an $n$-dimensional non-singular complete intersection in $\mathbb{CP}^{n+k}$ cut out by non-singular hypersurfaces of degrees $d_1, \dots , d_k$ intersecting transversally. Suppose that $n \ge 2$. Then $X_n(\underline{d})$ is simply-connected by the Lefschetz hyperplane theorem. By the adjunction formula the canonical bundle of $X_n(\underline{d})$ is the restriction of $\mathcal{O}( -n-k-1 + \sum_i d_i )$ to $X_n(\underline{d})$.

\begin{lemma}
We have that $X_n(\underline{d})$ is spin if and only if $-n-k-1+\sum_i d_i$ is even, in which case there is a unique spin structure. The corresponding square root of the canonical bundle is the restriction to $X_n(\underline{d})$ of $\mathcal{O}(m)$, where $2m = -n-k-1+\sum_i d_i$.
\end{lemma}
\begin{proof}
If $n>2$, then by the Lefschetz hyperplane theorem $H^2( X_n(\underline{d}) ; \mathbb{Z}) \cong \mathbb{Z}$ and is generated by the first Chern class of the restriction of $\mathcal{O}(1)$. It follows easily that the canonical bundle has a square root if and only if $-n-k-1+\sum_i d_i$ is even. In this case such a square root is unique and is clearly given by the restriction to $X_n(\underline{d})$ of $\mathcal{O}(m)$, where $2m = -n-k-1+\sum_i d_i$. The result still holds for $n=2$ (see \cite[Example 1.3.13 (d)]{gost}).
\end{proof}

\begin{corollary}
If $n>1$ is odd then $X_n(\underline{d})$ is spin if and only if an even number of the $d_i$ are even. If $n$ is even then $X_n(\underline{d})$ is spin if and only if an odd number of $d_i$ are even.
\end{corollary}
\begin{proof}
If $n$ is odd then $-n-k-1+\sum_i d_i = -n-1 + \sum_i (d_i-1)$ is even if and only if $\sum_i (d_i-1)$ is even, which holds if and only if an even number of the $d_i$ are even. The argument for even $n$ is similar.
\end{proof}

Now we assume that $n = 1 \; ({\rm mod} \; 4)$, $n \ge 5$ unless stated otherwise. Then the underlying smooth manifold of $X_n(\underline{d})$ has dimension $2n = 2 \; ({\rm mod} \; 8)$. If $X_n(\underline{d})$ is spin, then it has a unique spin structure and this determines a $KO$-orientation and a pushforward map
\[
\rho_* : KO^0( X_n(\underline{d}) ) \to KO^{-2}( pt ) \cong \mathbb{Z}_2
\]
where $\rho$ is the unique map to $pt$. By definition, the $\alpha$ invariant of $X_n(\underline{d})$ is
\[
\alpha( X_n(\underline{d}) ) = \rho_*(1) \in \mathbb{Z}_2,
\]
where $1$ denotes the class of the trivial rank $1$ bundle in $KO^0( X_n(\underline{d}) )$. Since the spin structure on $X_n(\underline{d})$ is unique, this is a diffeomorphism invariant and hence depends only on $n,d_1,\dots , d_k$. Thus we will write
\[
\alpha_n( \underline{d} ) = \alpha( X_n(\underline{d}) ) \in \mathbb{Z}_2,
\]
which is defined whenever $n = 1 \; ({\rm mod} \; 4)$, $n \ge 5$ and an even number of $d_i$ are even.

\begin{remark}
It is also possible to define $\alpha_n(\underline{d})$ when $n=1$ and an even number of $d_i$ are even. In this case $X_1(\underline{d})$ is always spin and usually has more than one spin structure. However, the assumption that an even number of $d_i$ are even determines a unique square root of the canonical bundle which extends to a line bundle on $\mathbb{CP}^{1+k}$. We can then define $\alpha_1(\underline{d})$ to be the alpha invariant of this particular spin structure.
\end{remark}

For the rest of this section it will be convenient to assume that $k$ is a multiple of $4$. This can always be arranged by possibly adding to $\underline{d}$ additional equations of degree $1$. This implies that $n+k$ is odd and in particular that $\mathbb{CP}^{n+k}$ has a (unique) spin structure. This spin structure determines orientations in complex and real $K$-theory, hence pushforward or index maps 
\[
\pi_*^{n+k} : K^0( \mathbb{CP}^{n+k} ) \to K^0(pt) \cong \mathbb{Z}, \quad \pi_*^{n+k} : KO^0( \mathbb{CP}^{n+k} ) \to KO^{-2}(pt) \cong \mathbb{Z}_2.
\]

More generally, suppose $N = 1 \; ({\rm mod} \; 4)$ so that $\mathbb{CP}^N$ has a (unique) spin structure. Let $\xi = [ \mathcal{O}(1)] \in K^0( \mathbb{CP}^N )$ be the class represented by the line bundle $\mathcal{O}(1)$. Then $K^0(\mathbb{CP}^N) \cong \mathbb{Z}[\xi,\xi^{-1}]/( (\xi - 1)^{N+1} )$. 

For any space $Y$, let $r : K^*(Y) \to KO^*(Y)$ denote the map corresponding to taking the underlying real bundle of a complex vector bundle. The pushforward maps $\pi_*^{N}$ are compatible with $r$ in the sense that we have a commutative diagram
\begin{equation*}\xymatrix{
K^0( \mathbb{CP}^{N}) \ar[d]^-r \ar[r]^-{\pi_*^{N}} & K^0(pt) \cong \mathbb{Z} \ar[d]^-{ {\rm mod} \; 2 } \\
KO^0( \mathbb{CP}^{N}) \ar[r]^-{\pi_*^{N}} & KO^{-2}(pt) \cong \mathbb{Z}_2 \\
}
\end{equation*}

Let $j : X_n(\underline{d}) \to \mathbb{CP}^{n+k}$ denote the inclusion map. Since $X_n(\underline{d})$ and $\mathbb{CP}^{n+k}$ are $KO$-oriented, we immediately have
\[
\alpha_n(\underline{d}) = \rho_*(1) = \pi_*^{n+k}( j_*(1)).
\]
The class $j_*(1) \in KO^0(\mathbb{CP}^{n+k})$ can be computed using the Thom isomorphism. More precisely, let $U$ denote a tubular neighbourhood of $X$, which inherits a spin structure from $\mathbb{CP}^{n+k}$. The inclusion $j : X_n(\underline{d}) \to \mathbb{CP}^{n+k}$ factors as
\[
\xymatrix{
X_n(\underline{d}) \ar[rr]^-j \ar[dr]_-{\zeta} & & \mathbb{CP}^{n+k} \\
& U \ar[ur]_-{\iota} &
}
\]
where $\zeta$ and $\iota$ are the obvious inclusions. Then $j_*$ may be defined as the composition of the Thom isomorphism and the excision isomorphism as follows:
\[
\xymatrix{
KO^0( X_n(\underline{d})) \ar[d]^-{(\text{Thom iso})}_{\zeta_*} \ar[rr]^-{j_*} & & KO^0(\mathbb{CP}^{n+k}) \\
KO^0( U , U - X_n(\underline{d})) \ar[rr]_-{(\text{Excision})} & & KO^0( \mathbb{CP}^{n+k} , \mathbb{CP}^{n+k} - X_n(\underline{d}) ) \ar[u]_-{\iota_*}
}
\]

Let $V(\underline{d}) = \mathcal{O}(d_1) \oplus \cdots \oplus \mathcal{O}(d_k)$. Since $k$ is even and an even number of $d_i$ are even, we see that $det( V(\underline{d_i}) )$ is the square root of a unique line bundle on $\mathbb{CP}^{n+k}$. This uniquely determines a spin structure on $V(\underline{d})$. Since the complex rank of $V(\underline{d})$ is even it has positive and negative spinor bundles $S_+(\underline{d}), S_-(\underline{d})$. Moreover, since $k$ is a multiple of $4$, $V(\underline{d})$ has real rank a multiple of $8$, hence the spinor bundle $S_+(\underline{d}), S_-(\underline{d})$ have real structures.

Suppose that $X_n(\underline{d})$ is defined by polynomials $f_1, \dots , f_k$ of degrees $d_1, \dots , d_k$. Then $f = (f_1, \dots , f_k)$ defines a section of $V(\underline{d})$ vanishing precisely on $X_n(\underline{d})$. Let
\[
c_f : S_+(\underline{d}) \to S_-(\underline{d})
\]
denote Clifford multiplication by $f$. This is an isomorphism away from $X_n(\underline{d})$, hence $( c_f : S_+(\underline{d}) \to S_-(\underline{d} ))$ defines a class in $KO^0( U , U - X_n(\underline{d})) \cong KO^0( \mathbb{CP}^{n+k} , \mathbb{CP}^{n+k} - X_n(\underline{d}) )$, which by the definition of the Thom isomorphism is preciely $\zeta_*(1)$. Then
\[
j_*(1) = \iota_*( c_f : S_+(\underline{d}) \to S_-(\underline{d}) ) = S_+(\underline{d}) - S_-(\underline{d}) \in KO^0( \mathbb{CP}^{n+k})
\]
is the underlying difference bundle. So we have shown:
\[
\alpha_n(\underline{d}) = \pi_*^{n+k}\left( S_+(\underline{d}) - S_-(\underline{d}) \right).
\]

The underlying complex virtual vector bundle $S_+(\underline{d}) - S_-(\underline{d})$ is easily seen to be
\begin{equation}\label{equ:xiterms}
\bigoplus_{\underline{\epsilon} = (\epsilon_1 , \dots , \epsilon_k) \atop \epsilon_1 , \dots , \epsilon_k = \pm 1} sgn( \underline{\epsilon} ) \, \xi^{ (\underline{\epsilon} \cdot \underline{d}) / 2}
\end{equation}
where $sgn(\underline{\epsilon}) = \epsilon_1 \cdots \epsilon_k$ and $\underline{\epsilon} \cdot \underline{d} = \epsilon_1 d_1 + \cdots + \epsilon_k d_k$. Note that $\underline{\epsilon} \cdot \underline{d}$ is always even since $k$ is even and an even number of $d_i$ are even.

Observe that (\ref{equ:xiterms}) is a Laurent polynomial in $\xi$ and is invariant under $\xi \mapsto \xi^{-1}$. The underlying real structure on $S_+(\underline{d}) - S_-(\underline{d})$ is obtained by pairing off monomials $\xi^m$ and $\xi^{-m}$ in (\ref{equ:xiterms}) and interpreting $\xi^m + \xi^{-m}$ as the underlying real rank $2$ vector bundle of $\mathcal{O}(m)$. To see that this gives the correct real structure on $S_+(\underline{d}) - S_-(\underline{d})$, view $\mathbb{CP}^{n+k}$ as the quotient of the unit sphere $S^{2n+2k+1} \subset \mathbb{C}^{n+k+1}$ by the natural $S^1$-action. Then $\mathcal{O}(1)$ is the pullback under $S^{2n+2k+1} \to pt$ of the $S^1$-equivariant line bundle generating $H^2_{S^1}(pt ; \mathbb{Z})$. It follows that $S_+(\underline{d}) - S_-(\underline{d})$ with its natural real structure is similarly a pullback from $KO_{S^1}(pt)$. But the complexification map $c : KO_{S^1}(pt) \to K_{S^1}(pt)$ is injective, so there is a unique way to lift a class in the image of $c$ to $KO_{S^1}(pt)$.

The terms in (\ref{equ:xiterms}) can be paired off by taking pairs of the form $\underline{\epsilon} = (1 , \epsilon_2, \dots , \epsilon_k)$, $-\underline{\epsilon} = (-1 , -\epsilon_2 , \dots , \epsilon_k)$, giving
\[
\bigoplus_{\underline{\epsilon} = (1 , \epsilon_2 , \dots , \epsilon_k) \atop \epsilon_2 , \dots , \epsilon_k = \pm 1} sgn( \underline{\epsilon} ) \, \left( \xi^{ (\underline{\epsilon} \cdot \underline{d}) / 2} + \xi^{ -(\underline{\epsilon} \cdot \underline{d}) / 2} \right).
\]
Note that $sgn(-\underline{\epsilon}) = sgn( \underline{\epsilon})$ as $k$ is even. If follows that:
\[
j_*(1) = S_+(\underline{d}) - S_-(\underline{d}) = r(\beta) \in KO^0(\mathbb{CP}^{n+k})
\]
where
\[
\beta = \bigoplus_{\underline{\epsilon} = (1 , \epsilon_2 , \dots , \epsilon_k) \atop \epsilon_2 , \dots , \epsilon_k = \pm 1} sgn( \underline{\epsilon} ) \, \xi^{ (\underline{\epsilon} \cdot \underline{d}) / 2} \in K^0(\mathbb{CP}^{n+k}).
\]

Now since $\alpha_n(\underline{d}) = \pi_*^{n+k}( j_*(1)) = \pi_*^{n+k}( r(\beta) ) = \pi_*^{n+k}( \beta) \; ({\rm mod} \; 2)$, we have shown:

\begin{proposition}\label{prop:alphan}
We have
\[
\alpha_n(\underline{d}) = \pi_*^{n+k}(\beta) \; ({\rm mod} \; 2)
\]
where
\begin{equation}\label{equ:beta}
\beta = \bigoplus_{\underline{\epsilon} = (1 , \epsilon_2 , \dots , \epsilon_k) \atop \epsilon_2 , \dots , \epsilon_k = \pm 1} sgn( \underline{\epsilon} ) \, \xi^{ (\underline{\epsilon} \cdot \underline{d}) / 2} \in K^0(\mathbb{CP}^{n+k}).
\end{equation}
\end{proposition}

\section{Pontryagin classes}\label{sec:pont}

Let $X = X_n(\underline{d})$ be a complete intersection. We denote by $x \in H^2(X ; \mathbb{Z})$ the first Chern class of the restriction of $\mathcal{O}(1)$. If $n > 2$, the Lefschetz hyperplane theorem implies that $x$ is a generator of $H^2(X ; \mathbb{Z})$. Let $d_{tot} = d_1 d_2 \cdots d_k$ be the total degree of $X$. Then it follows easily that $x^n$ is $d_{tot}$ times a generator of $H^{2n}( X ; \mathbb{Z})$, hence $d_{tot}$ is a diffeomorphism invariant (provided $n>2$).

Let $Q_n$ denote the subring of $H^*(X ; \mathbb{Q})$ generated by $x$ over $\mathbb{Q}$, so $Q_n \cong \mathbb{Q}[x]/\langle x^{n+1} \rangle$. We will show that the Pontryagin classes of $X$ belong to $Q_n$. This makes it possible to compare Pontryagin classes of different complete intersections of the same dimension, since $Q_n$ depends only on $n$. 

Let $N$ denote the normal bundle of $X$ in $\mathbb{CP}^{n+k}$. Since $X$ is a complete intersection the normal bundle is given by the restriction of $V(\underline{d}) = \mathcal{O}(d_1) \oplus \cdots \oplus \mathcal{O}(d_k)$ to $X$. From the exact sequence
\[
0 \to TX \to T\mathbb{CP}^{n+k}|_X \to N \to 0
\]
and the Euler sequence
\[
0 \to \mathbb{C} \to \mathcal{O}(1)^{\oplus n+k+1} \to T\mathbb{CP}^{n+k} \to 0
\]
One finds that the total Chern class of $X$ is
\begin{equation}\label{equ:chern}
c(TX) = (1+x )^{n+k+1} ( 1+d_1 x)^{-1} \cdots (1+d_k x)^{-1}.
\end{equation}
Similarly, the total Pontryagin class of $X$ is given by:
\begin{equation}\label{equ:pont}
p(TX) = \left( 1+x^2 \right)^{n+k+1} ( 1+d_1^2 x^2)^{-1} \cdots (1+d_k^2 x^2)^{-1}.
\end{equation}

For $j \ge 1$ let
\[
e_j = \sum_{i_1 < \cdots < i_j } d_{i_1} \cdots d_{i_j}
\]
be the $j$-th elementary symmetric polynomial in $(d_1,\dots , d_k)$ and
\[
\sigma_j = \sum_i d_i^j
\]
the $j$-th power sum. We also set $e_0 = 1$. Recall the identity
\[
(1+d_1 x) \cdots (1+d_k x) = \sum_{j \ge 0} e_j x^j = exp \left( \sum_{j \ge 1} \frac{(-1)^{j+1}}{j} \sigma_n x^j \right),
\]
in the ring $\mathbb{Q}[[x]]$ of formal power series over $\mathbb{Q}$. Of course the same relation also holds in the ring $Q_n$. A special case of this relation is 
\[
(1+x) = exp( log(1+x) ) = exp \left( \sum_{j \ge 1} \frac{(-1)^{j+1}}{j} x^j \right).
\]
From this and Equation (\ref{equ:chern}), we find
\[
c(TX) = (1+x)^{n+1} exp \left( \sum_{j \ge 1} \frac{(-1)^j}{j} (\sigma_j - k ) x^j \right).
\]
Similarly, from Equation (\ref{equ:pont}), we find
\[
p(TX) = (1+x^2)^{n+1} exp \left( \sum_{j \ge 1} \frac{(-1)^j}{j} (\sigma_{2j} - k ) x^{2j} \right).
\]

Hence we have proven the following:
\begin{proposition}\label{prop:pontcl}
We have:
\begin{itemize}
\item{The Chern classes of $TX$ can be expressed in terms of $n$ and $(\sigma_j - k)$ for $1 \le j \le n$.}
\item{The Pontryagin classes of $TX$ can be expressed in terms of $n$ and $(\sigma_{2j}-k)$ for $1 \le j \le \lfloor n/2 \rfloor $.}
\end{itemize}
\end{proposition}

\begin{corollary}
Let $X = X_n(\underline{d}), X' = X_n(\underline{d}')$ be two complete intersections of the same dimension for multi-degrees $\underline{d} = (d_1, \dots , d_k)$ and $\underline{d}' = (d'_1 , \dots , d'_{k'})$. If $k=k'$ then $X$ and $X'$ have the same Pontryagin classes (considered as elements of $Q_n$) if and only if $\sum_i d_i^{2j} = \sum_i {d'_i}^{2j}$ for all $j$ with $1\le j \le n/2$.
\end{corollary}

\begin{remark}
Suppose we are given any two multi-degrees $\underline{d}, \underline{d}'$ defining complete intersections $X = X_n(\underline{d}) \subset \mathbb{CP}^{n+k}$ and $X' = X_n(\underline{d}') \subset \mathbb{CP}^{n+k'}$. If $k' < k$, then we can view $\mathbb{CP}^{n+k'}$ as a projective subspace of $\mathbb{CP}^{n+k}$ and view $X'$ as a complete intersection in $\mathbb{CP}^{n+k}$. In other words, we can always assume that $k=k'$ by possibly adding in some equations of degree $1$.
\end{remark}

\section{Proof of the main theorem}\label{sec:main}

Let $n\ge 1$ be an integer with $n = 1 \; ({\rm mod} \; 4)$ and let $d_1, \dots , d_k$ be positive integers and assume that an even number of the $d_i$ are even. We write $\underline{d} = (d_1, \dots , d_k)$. We assume that $k$ is a multiple of $4$, which can always be arranged by possibly adding to $\underline{d}$ additional equations of degree $1$. Then the complete intersection $X_n(\underline{d})$ has a distinguished spin structure and from Proposition \ref{prop:alphan} we have
\[
\alpha_n(\underline{d}) = \pi_*^{n+k}(\beta),
\]
where $\beta \in K^0(\mathbb{CP}^{n+k})$ is given by Equation (\ref{equ:beta}).

Consider now the Chern character:
\begin{equation*}
\begin{aligned}
Ch : K^0( \mathbb{CP}^{n+k} ) &\to Q_{n+k} = \mathbb{Q}[x]/( x^{n+k+1} ) \\
\xi &\mapsto e^x.
\end{aligned}
\end{equation*}
Since $K^0( \mathbb{CP}^{n+k})$ has no torsion, $Ch$ is injective. Therefore the class $\beta \in K^0( \mathbb{CP}^{n+k})$ is completely determined by $Ch(\beta) \in Q_{n+k}$. On the other hand Proposition \ref{prop:alphan} shows that $\beta$ determines $\alpha_n(\underline{d})$. Thus the alpha invariant is completely detemined by $Ch(\beta)$.

\begin{lemma}\label{lem:chbeta}
Let $n\ge 1$ be an integer with $n = 1 \; ({\rm mod} \; 4)$ and let $d_1, \dots , d_k$ be positive integers and assume that an even number of the $d_i$ are even. Assume that $k$ is a multiple of $4$ and that $d_1=d_2=1$ (this can always be accomplished by adding terms of degree $1$ to $\underline{d}$). Let $\beta$ be defined as in Equation (\ref{equ:beta}). Then
\[
Ch(\beta) = d_{tot} x^{k-1} \left( \frac{e^x-1}{x} \right) \left( \frac{x}{e^{x/2}-e^{-x/2}} \right)^2 \prod_{i=1}^k \left( \frac{e^{d_ix/2} - e^{-d_ix/2}}{d_i x} \right) \in Q_{n+k},
\]
where $d_{tot} = d_1 \cdots d_k$ is the total degree.
\end{lemma}
\begin{proof}
Applying the Chern character to Equation (\ref{equ:beta}), we have
\begin{equation*}
\begin{aligned}
Ch(\beta) &= \sum_{\epsilon_2 , \dots , \epsilon_k = \pm 1} sgn(\underline{\epsilon}) \, e^{ (\underline{\epsilon} \cdot \underline{d} )x/2} \\
&= ( e^{(d_1+d_2)x/2} - e^{(d_1-d_2)x/2} ) \prod_{i=3}^k ( e^{d_ix/2} - e^{-d_ix/2} ) \\
&= (e^x - 1) \prod_{i=3}^k (e^{d_ix/2} - e^{-d_ix/2} ) \\
&= d_3 \cdots d_k x^{k-2} (e^x-1) \prod_{i=3}^k \left( \frac{ e^{d_ix/2} - e^{-d_ix/2} }{d_ix} \right) \\
&= d_{tot} x^{k-1} \left( \frac{e^x-1}{x} \right) \left( \frac{x}{e^{x/2}-e^{-x/2}} \right)^2 \prod_{i=1}^k \left( \frac{e^{d_ix/2} - e^{-d_ix/2}}{d_i x} \right).
\end{aligned}
\end{equation*}

\end{proof}

Lemma \ref{lem:chbeta} can be simplified to
\begin{equation}\label{equ:chbeta2}
Ch(\beta) = d_{tot} x^{k-1} \left( \frac{e^x-1}{x} \right) \left( \frac{x}{e^{x/2}-e^{-x/2}} \right)^2 \hat{A}( \underline{d} )^{-1} \in Q_{n+k},
\end{equation}
where $\hat{A}( \underline{d} )$ denotes the $A$-hat multiplicative sequence applied to $d_1x, \dots , d_k x$:
\[
\hat{A}( \underline{d} ) = \prod_{i=1}^k \left( \frac{d_ix}{e^{d_ix/2} - e^{-d_ix/2}} \right) \in \mathbb{Q}[[x]].
\]

Note that since $Ch(\beta) \in Q_{n+k}$ and since there is a factor of $x^{k-1}$ in the right hand side of Equation (\ref{equ:chbeta2}), it follows that $Ch(\beta)$ (and hence $\alpha_n(\underline{d})$) depends only on $n,k$, $d_{tot}$ and the expansion of $\hat{A}( \underline{d})$ up to order $x^{n+1}$. Moreover, expanding $\hat{A}(\underline{d})$ in the form
\[
\hat{A}(\underline{d}) = 1 + \hat{A}_1(\underline{d})x^2 + \hat{A}_2(\underline{d}) x^4 + \cdots
\]
we have that the coefficient $\hat{A}_j(\underline{d})$ is a symmetric polynomial in $d_1^2, \dots  , d_k^2$ of degree $j$ (i.e. of degree $2j$ in $d_1, \dots , d_k$). Therefore, $\alpha_n(\underline{d})$ is completely determined by $n,k$, $d_{tot}$ and the power sums $\sigma_{2j} = \sum_i d_i^{2j}$ for $1 \le j \le (n+1)/2$. According to Proposition \ref{prop:pontcl}, for given $n$ and $k$, the power sums $\sigma_2 , \sigma_4 ,  \dots , \sigma_{n-1}$ are determined by the Pontryagin classes of $X = X_n(\underline{d})$. However this leaves $\sigma_{n+1}$ which is {\em not} determined by the Pontryagin classes of $X$. In what follows, we seek to isolate the possible dependence of $\alpha_n(\underline{d})$ on $\sigma_{n+1}$. 

\begin{lemma}\label{lem:bernoulli}
For any $j \ge 1$, view $\hat{A}_j(d_1, \dots , d_k) \in \mathbb{Q}[\sigma_2 , \sigma_4 , \dots ]$ as a symmetric polynomial in the even power sums. Then
\[
\hat{A}_j(d_1 , \dots , d_k) = (-1)^j \frac{B_{2j}}{2 \cdot (2j)! } \sigma_{2j} \; ({\rm mod} \;  \sigma_2, \dots , \sigma_{2(j-1)} )
\]
where $B_{2j}$ is the $2j$-th Bernoulli number. Similarly, if we write
\[
\hat{A}(\underline{d})^{-1} = 1 + \hat{A}'_1(\underline{d})x^2 + \hat{A}'_2(\underline{d}) x^4 + \cdots
\]
for some symmetric polynomials $A'_j( \underline{d}) \in \mathbb{Q}[\sigma_2 , \sigma_4 , \dots ]$, then
\[
\hat{A}'_j(d_1 , \dots , d_k) = (-1)^{j+1} \frac{B_{2j}}{2 \cdot (2j)! } \sigma_{2j} \; ({\rm mod} \;  \sigma_2, \dots , \sigma_{2(j-1)} ).
\]
\end{lemma}
\begin{proof}
Recall that $\hat{A}(\underline{d})$ is the multiplicative sequence associated to the power series
\[
Q(z) = \frac{ \sqrt{z}/2}{ {\rm sinh}{( \sqrt{z}/2} )}.
\]
Of course this also means that $\hat{A}(\underline{d})^{-1}$ is the multiplicative sequence associated to the power series
\[
Q(z)^{-1} = \frac{ {\rm sinh}{( \sqrt{z}/2} )}{\sqrt{z}/2}.
\]
Since $\hat{A}_j(d_1, \dots , d_k)$ is a symmetric polynomial in $d_1^2, \dots , d_k^2$ of degree $j$, it must have the form
\[
\hat{A}_j = s_j \sigma_{2j} \; ({\rm mod} \; \sigma_2 , \dots , \sigma_{2(j-1)})
\]
for some $s_j \in \mathbb{Q}$. From \cite[\textsection 1]{hir}, we have the identity
\[
1 - z \frac{Q'(z)}{Q(z)} = \sum_{j=0}^\infty (-1)^j s_j z^j.
\]
A short calculation gives
\begin{equation*}
\begin{aligned}
1 - z \frac{Q'(z)}{Q(z)} &= 1 - z\left( \frac{1}{2z} - \frac{1}{2\sqrt{z}} \frac{ {\rm cosh}( \sqrt{z}/2 ) }{ {\rm sinh}( \sqrt{z}/2 ) } \right) \\
&= \frac{1}{2} + \frac{1}{2}\frac{\sqrt{z}/2}{{\rm tanh}(\sqrt{z}/2)} \\
&= 1 + \sum_{j=1}^\infty \frac{B_{2j}}{2 \cdot (2j)!} z^j
\end{aligned}
\end{equation*}
where $B_{2j}$ denotes the $2j$-th Bernoulli number. Hence $s_j = (-1)^j B_{2j}/(2 \cdot (2j)!)$. The corresponding computation for $\hat{A}^{-1}$ follows immediately from $\frac{ Q^{-1}(z)'}{Q^{-1}(z)} = - \frac{Q(z)'}{Q(z)}$.
\end{proof}

For any non-zero integer $m$, let $\nu_2(m)$ be the number of times $2$ divides $m$. We note here that for any $m \ge 0$, we have (\cite[Chapter 2, Exercise 6]{ir}):
\[
\nu_2( m!) = \lfloor m/2 \rfloor + \lfloor m/4 \rfloor + \lfloor m/8 \rfloor + \cdots.
\]

\begin{lemma}
Let $X = X_n(\underline{d}), X' = X_n(\underline{d}')$, where $\underline{d} = (d_1 , \dots , d_k)$, $\underline{d}' = (d'_1 , \dots , d'_{k'})$ be two spin complete intersections of dimension $n \ge 1$, $n = 1 \; ({\rm mod} \; 4)$. Suppose that $k=k'$ and that $X,X'$ have the same Pontryagin classes and the same total degree. Then
\[
\alpha_n(\underline{d}) - \alpha_n(\underline{d}') = \frac{d_{tot} ( \sum_i d_i^{n+1} - \sum_i {d'_i}^{n+1} )}{ 2^{\rho} } \; ({\rm mod} \; 2)
\]
where $\rho = 2 + \nu_2( (n+1)! )$.
\end{lemma}
\begin{proof}
After possibly adding some equations of degree $1$, we can assume that $k$ is a multiple of $4$. We will temporarily assume that $d_1=d_2=d_1'=d_2'=1$. Later we will see how to remove this assumption. Define $\beta, \beta' \in K^0( \mathbb{CP}^{n+k})$ by:
\[
\beta = \bigoplus_{\epsilon_2 , \dots , \epsilon_k = \pm 1} sgn(\underline{\epsilon}) \, \xi^{ (\underline{\epsilon} \cdot \underline{d} )/2}, \quad \quad \beta' = \bigoplus_{\epsilon_2 , \dots , \epsilon_k = \pm 1} sgn(\underline{\epsilon}) \, \xi^{ (\underline{\epsilon} \cdot \underline{d}' )/2}.
\]
Let $\theta = \beta - \beta'$. Then from Equation (\ref{equ:chbeta2}), we have
\[
Ch(\theta) = d_{tot} x^{k-1} \left( \frac{e^x-1}{x} \right) \left( \frac{x}{e^{x/2}-e^{-x/2}} \right)^2 \left( \hat{A}( \underline{d} )^{-1} - \hat{A}(\underline{d}')^{-1} \right) \in Q_{n+k}.
\]
But since $X,X'$ have the same Pontryagin classes, Lemma \ref{lem:bernoulli} implies that
\[
\hat{A}(\underline{d})^{-1} - \hat{A}(\underline{d}')^{-1} = - \frac{B_{n+1}}{2 \cdot (n+1)! } (\sigma_{n+1}(\underline{d}) - \sigma_{n+1}(\underline{d}')) x^{n+1} + \cdots
\]
where $+\cdots$ represents terms involving higher powers of $x$ and we used $n = 1 \; ({\rm mod} \; 4)$ to see that $(-1)^{(n+1)/2} = -1$. Substituting into the formula for $Ch(\theta)$ and noting that $x^{n+k+1}=0$ in $Q_{n+k}$, we get
\[
Ch(\theta) = -d_{tot} x^{n+k} \left( \frac{e^x-1}{x} \right) \left( \frac{x}{e^{x/2}-e^{-x/2}} \right)^2 \frac{B_{n+1}}{2 \cdot (n+1)! } (\sigma_{n+1}(\underline{d}) - \sigma_{n+1}(\underline{d}')).
\]
But $\frac{e^x-1}{x} = 1 + \cdots$, where $+\cdots$ represents terms involving positive powers of $x$. Similarly $\left( \frac{x}{e^{x/2}-e^{-x/2}} \right)^2 = 1 + \cdots $, so that
\[
Ch(\theta) = -d_{tot} x^{n+k} \frac{B_{n+1}}{2 \cdot (n+1)! } (\sigma_{n+1}(\underline{d}) - \sigma_{n+1}(\underline{d}')).
\]
Then since $Ch : K^0(\mathbb{CP}^{n+k}) \to Q_{n+k}$ is injective, it follows that
\[
\theta = -\frac{d_{tot} \cdot B_{n+1}}{2 \cdot (n+1)! } (\sigma_{n+1}(\underline{d}) - \sigma_{n+1}(\underline{d}')) (\xi - 1)^{n+k} \in K^0(\mathbb{CP}^{n+k}).
\]
Then by Proposition \ref{prop:alphan}, we have
\begin{equation*}
\begin{aligned}
\alpha_n(\underline{d}) - \alpha_n(\underline{d}') &= \pi_*^{n+k}( \theta) \\
&= \frac{ d_{tot} \cdot B_{n+1}}{2 \cdot (n+1)!} (\sigma_{n+1}(\underline{d}) - \sigma_{n+1}(\underline{d}')) \, \pi_*^{n+k}( (\xi-1)^{n+k} ) \; ({\rm mod} \; 2).
\end{aligned}
\end{equation*}
By the Atiyah-Singer index theorem, we have
\begin{equation*}
\begin{aligned}
\pi_*^{n+k}( (\xi-1)^{n+k} ) &= \int_{\mathbb{CP}^{n+k}} Ch( (\xi-1)^{n+k} ) \, \hat{A}( T\mathbb{CP}^{n+k} ) \\
&= \int_{\mathbb{CP}^{n+k}} x^{n+k} \, \hat{A}( T\mathbb{CP}^{n+k} ) \\
&= 1
\end{aligned}
\end{equation*}
and therefore
\[
\alpha_n(\underline{d}) - \alpha_n(\underline{d}') = \frac{ d_{tot} \cdot B_{n+1}}{2 \cdot (n+1)!} (\sigma_{n+1}(\underline{d}) - \sigma_{n+1}(\underline{d}')) \; ({\rm mod} \; 2).
\]
To complete the proof, write $B_{n+1}$ in the form $B_{n+1} = \frac{u}{v}$, where $u,v$ are coprime integers. By the Von Staudt-Clausen theorem we have that $2$ divides the denominator of $B_{n+1}$ exactly once, so we can write $v = 2v'$ with $v'$ odd. Moreover $u$ is odd since $u$ and $v$ are coprime. Let us also write $(n+1)!$ as $(n+1)! = 2^{a}b$, where $a = \nu_2( (n+1)!)$ and $b$ is odd. Then
\begin{equation*}
\begin{aligned}
\alpha_n(\underline{d}) - \alpha_n(\underline{d}') &= bv' (\alpha_n(\underline{d}) - \alpha_n(\underline{d}')) \; ({\rm mod} \; 2) \\
&= bv' \frac{ d_{tot} \cdot B_{n+1}}{2 \cdot (n+1)!} (\sigma_{n+1}(\underline{d}) - \sigma_{n+1}(\underline{d}')) \; ({\rm mod} \; 2) \\
&= bv' \frac{ d_{tot} \cdot u }{2^{a+2} bv'} (\sigma_{n+1}(\underline{d}) - \sigma_{n+1}(\underline{d}')) \; ({\rm mod} \; 2) \\
&= u \frac{ d_{tot} (\sigma_{n+1}(\underline{d}) - \sigma_{n+1}(\underline{d}'))}{ 2^\rho }  \; ({\rm mod} \; 2) \\
&= \frac{ d_{tot} (\sigma_{n+1}(\underline{d}) - \sigma_{n+1}(\underline{d}'))}{ 2^\rho }  \; ({\rm mod} \; 2).
\end{aligned}
\end{equation*}
This proves the result under the additional assumption that $d_1=d_2 = d'_1 = d'_2 = 1$. However the result still holds even without this assumption. To see this note that we can always replace $k$ by $k+4$ and add to both $\underline{d}$ and $\underline{d}'$ four additional equations of degree $1$. Then after re-ordering we have $d_1=d_2 = d'_1 = d'_2 = 1$. But this operation does not change the value of $d_{tot} (\sigma_{n+1}(\underline{d}) - \sigma_{n+1}(\underline{d}'))$.
\end{proof}

\begin{corollary}\label{cor:divis}
Let $X = X_n(\underline{d}), X' = X_n(\underline{d}')$, where $\underline{d} = (d_1 , \dots , d_k)$, $\underline{d}' = (d'_1 , \dots , d'_{k'})$ be two spin complete intersections of dimension $n \ge 1$, $n = 1 \; ({\rm mod} \; 4)$. Suppose that $k=k'$ and that $X,X'$ have the same Pontryagin classes and the same total degree. If $d_{tot} (\sigma_{n+1}(\underline{d}) - \sigma_{n+1}(\underline{d}'))$ is divisible by $2^{\rho+1}$, where $\rho = 2 + \nu_2( (n+1)! )$, then $\alpha_n(\underline{d}) = \alpha_n( \underline{d}')$.
\end{corollary}

\begin{theorem}
Let $X = X_n(\underline{d}), X' = X_n(\underline{d}')$, where $\underline{d} = (d_1 , \dots , d_k)$, $\underline{d}' = (d'_1 , \dots , d'_{k'})$ be two spin complete intersections of dimension $n \ge 5$, $n = 1 \; ({\rm mod} \; 4)$. Suppose that $X,X'$ have the same Pontryagin classes and the same total degree. Then $\alpha_n(\underline{d}) = \alpha_n(\underline{d}')$. Hence the alpha invariant of a spin complete intersection $X$ depends only on the dimension, total degree and Pontryagin classes of $X$.
\end{theorem}

\begin{remark}
Note that assumption $n \ge 5$ is not necessary. We will show in Section \ref{sec:n=1} that $\alpha_1(\underline{d})$ depends only on $d_{tot}$. A key difference however is that $d_{tot}$ is not a diffeomorphism invariant of $X_n(\underline{d})$ for $n=1$. In contrast $d_{tot}$ is determined by the cohomology ring of $X_n(\underline{d})$ if $n>2$.
\end{remark}

\begin{proof}
As usual, we may assume that $k=k'$ and that $k$ is divisible by $4$, since we can always include additional equations of degree $1$.

For each $i \ge 1$, let
\[
a_i = \# \{ j \; | \; d_j = i \}  -  \# \{ j \; | \; d'_j = i \} .
\]
The equality of Pontryagin classes of $X,X'$ together with the equality $k=k'$, gives
\begin{equation}\label{equ:vanish}
\sigma_{2j}(\underline{d}) - \sigma_{2j}(\underline{d}') = \sum_i i^{2j} a_i = 0, \text{ for } 0 \le j \le (n-1)/2.
\end{equation}
Note also that
\[
\sigma_{n+1}(\underline{d}) - \sigma_{n+1}(\underline{d}') = \sum_i i^{n+1} a_i.
\]
According to Corollary \ref{cor:divis}, the result will follow if we can show that $ d_{tot} \sum_i i^{n+1} a_i$ is divisible by $2^{\rho+1}$. From Equation (\ref{equ:vanish}), we have that
\[
\sum_i i^{n+1} a_i = \sum_i  f(i) a_i
\]
for any polynomial $f(i)$ of the form $f(i) = i^{n+1} + c_2 i^{n-1} + c_4 i^{n-3} + \cdots + c_{n+1}$ for some integers $c_2 , c_4 , \dots $. In other words, $f(i)$ is an even, monic polynomial of degree $n+1$ with integer coefficients.

Consider
\[
f_1(i) = \frac{1}{2} \left( p(i) + p(-i) \right), \text{ where } p(i) = i(i-1)(i-2) \cdots (i-n).
\]
Note that the coefficients of $f_1$ are integers. For any integer $i$, we have $p(i) = \binom{i}{n+1} (n+1)!$ is an integer multiple of $(n+1)!$, hence $2$ divides $p(i)$ at least $\nu_2( (n+1)!)$ times. A similar argument applies to $p(-i)$, hence for any integer $i$, $f_1(i)$ is divisible by $2$ at least $\nu_2( (n+1)!) - 1 = \rho-3$ times. Therefore, if $\nu_2( d_{tot} ) \ge 4$ we have that
\[
d_{tot} ( \sigma_{n+1}(\underline{d}) - \sigma_{n+1}(\underline{d}') ) = d_{tot} \sum_i f_1(i) a_i
\]
is divisible by $2^{\rho+1}$, so $\alpha_n(\underline{d}) = \alpha_n(\underline{d}')$ by Corollary \ref{cor:divis}.

It follows that we may restrict to the case $\nu_2( d_{tot}) \le 3$. Note also that $\nu_2( d_{tot}) = 1$ is impossible, since for $X$ and $X'$ to be spin, we need an even number of the $d_i$ and an even number of the $d_i'$ to be even. This leaves us with three cases according to whether $\nu_2( d_{tot}) = 0,2$ or $3$. We treat each of these cases in turn.

Suppose $\nu_2(d_{tot}) = 0$. So all the $d_i$ and $d'_i$ are odd. Consider
\[
f_2(i) = (i^2-1)(i^2-5^2)(i^2 - 5^4) \cdots (i^2 - 5^{n-1}).
\]
By \cite[Theorem 2$'$, Chapter 4]{ir}, for any $l \ge 3$ the subgroup 
\[
\{ j^2 \; | \; j \in \mathbb{Z}_{2^l}^* \} \subset \mathbb{Z}_{2^l}^*
\]
is cyclic of order $2^{l-3}$ and generated by $5^2$. Hence for any odd $i$, each of the $(n+1)/2$ factors
\[
(i^2-1), (i^2-5^2), (i^2 - 5^4),  \dots , (i^2 - 5^{n-1})
\]
is divisible by $8$, at least $\lfloor (n+1)/4 \rfloor$ of them are divisible by $16$, at least $\lfloor (n+1)/8 \rfloor $ of them by $32$, etc. Hence for any odd integer $i$, $f_2(i)$ is divisible by $2$ at least
\[
3(n+1)/2 + \lfloor (n+1)/4 \rfloor + \lfloor (n+1)/8 \rfloor + \cdots = (n+1) + \nu_2( (n+1)! ) \ge 6 + \nu_2( (n+1)!) > \rho + 1
\]
times. Therefore
\[
d_{tot} ( \sigma_{n+1}(\underline{d}) - \sigma_{n+1}(\underline{d}') ) = d_{tot} \sum_{i \; {\rm odd}} f_2(i) a_i
\]
is divisible by $2^{\rho+1}$, so $\alpha_n(\underline{d}) = \alpha_n(\underline{d}')$ by Corollary \ref{cor:divis}.

Suppose $\nu_2(d_{tot}) = 2$. Since an even number of $d_i$ are even, we must have (after re-ordering terms) $\nu_2(d_1) = \nu_2(d_2) = 1$ and $d_i$ odd for $i>2$. Similarly after re-ordering we must have $\nu_2(d'_1) = \nu(d'_2) = 1$ and $d'_i$ is odd for $i > 2$. Suppose first that $n \ge 17$. Consider
\[
f_3(i) = (i^2-d_1^2)(i^2 - d_2^2)(i^2 - {d_1'}^2)( i^2 - {d_2'}^2) (i^2-1)(i^2-5^2) \cdots (i^2 - 5^{n-9}).
\]
If $i$ is even and $a_i \neq 0$, then $i$ is one of $d_1,d_2,d'_1,d'_2$ and hence $f_3(i)=0$. Now suppose that $i$ is odd. Then repeating the same argument as in the previous case (but now with only $(n-7)/2$ factors), we see that $f(i)$ is divisible by $2$ at least
\[
3(n-7)/2 + \lfloor (n-7)/4 \rfloor + \lfloor (n-7)/8 \rfloor + \cdots = (n-7) + \nu_2( (n-7)! ) = \nu_2( (2n-14)! )
\]
times. Now since $n=1 \; ({\rm mod} \; 4)$, we have $\nu_2( (2n-14)! ) = 2 + \nu_2( (2n-15)! )$ and $\rho-1 = 1 + \nu_2( (n+1)! ) = 2 + \nu_2( n!)$. But $n \ge 17$ implies $2n-14 \ge n$ and hence $\nu_2( (2n-14)! ) \ge \rho-1$. So if $i$ is odd, $2$ divides $f(i)$ at least $\rho-1$ times and hence $2^{\rho+1}$ divides $d_{tot} f(i)$, since $\nu_2(d_{tot}) = 2$. This together with the fact that $f(i)a_i = 0$ for even $i$ implies that $2^{\rho+1}$ divides $d_{tot} ( \sigma_{n+1}(\underline{d}) - \sigma_{n+1}(\underline{d}') )$ and thus $\alpha_n(\underline{d}) = \alpha_n(\underline{d}')$ by Corollary \ref{cor:divis}.

If $\nu_2(d_{tot}) = 2$ and $n$ is one of $5,9,13$, then we use a similar argument using one of the following polynomials:
\begin{equation*}
\begin{aligned}
(i^2-2^2)(i^2-1)(i^2-5^2) & \; \; \text{for } n = 5, \\
(i^2-2^2)^2(i^2-1)(i^2-5^2)(i^2-5^4) & \; \; \text{for } n = 9, \\
(i^2-2^2)^3(i^2-1)(i^2-5^2)(i^2-5^4)(i^2-5^6) & \; \; \text{for } n = 13.
\end{aligned}
\end{equation*}
To see that this works, we just need to check that these polynomials are divisible by $2^{\rho-1}$ for all $i$ with $a_i \neq 0$. We have $\rho-1= 5,9$ and $12$ for $n=5,9$ and $13$ respectively. If $i$ is even and $a_i \neq 0$, then $i = 2 \; ({\rm mod} \; 4)$ and
\[
(i^2 - 2^2) = 4( (i/2)^2 - 1),
\]
which is divisible by $2^5$, since $i/2$ is odd. So $(i^2-2^2)^2$ is divisible by $2^{10}$ and $(i^2-2^2)^3$ by $2^{15}$. This covers the case when $i$ is even. When $i$ is odd, applying our usual argument to the remaining factors $(i^2 - 5^2) \cdots $, we easily verify divisibility by $2^{\rho-1}$. Hence we again conclude that $\alpha_n(\underline{d}) = \alpha_n(\underline{d}')$.

Suppose $\nu_2(d_{tot}) = 3$. Since an even number of $d_i$ are even, we must have (after re-ordering terms) $\nu_2(d_1) = 2$, $\nu_2(d_2) = 1$ and $d_i$ odd for $i>2$. Similarly after re-ordering we must have $\nu_2(d'_1) = 2$, $\nu(d'_2) = 1$ and $d'_i$ is odd for $i > 2$.

If $n \ge 17$, then consider once again
\[
f_3(i) = (i^2-d_1^2)(i^2 - d_2^2)(i^2 - {d_1'}^2)( i^2 - {d_2'}^2) (i^2-1)(i^2-5^2) \cdots (i^2 - 5^{n-9}).
\]
By the same argument as in the case $\nu_2(d_{tot}) = 2$, we find that $\alpha_n(\underline{d}) = \alpha_n(\underline{d}')$.

Now suppose that $\nu_3(d_{tot}) = 3$ and $n = 5, 9$ or $13$. Since $\nu_3(d_{tot}) = 3$, we can assume $d_1 = 4u_1, d_2 = 2u_2$ for some odd integers $u_1,u_2$ and that $d_i$ is odd for $i>2$. Similarly assume $d'_1 = 4u'_1$, $d'_2 = 2u'_2$ for some odd integers $u'_1,u'_2$ and that $d'_i$ is odd for $i>2$. For $n = 5,9$ or $13$, we will show that $\sigma_{n+1}(\underline{d}) - \sigma_{n+1}(\underline{d}')$ is divisible by $2$ at least $\nu_2( (n+1)! )$ times and hence $d_{tot} ( \sigma_{n+1}(\underline{d}) - \sigma_{n+1}(\underline{d}') )$ is divisible by $2$ at least $3 + \nu_2( (n+1)!) = \rho+1$ times. This will imply $\alpha_n(\underline{d}) = \alpha_n(\underline{d}')$.

In the case $n=5$, we have $\nu_2( (n+1)!) = \nu_2(6!) = 4$. Consider $f_4(i) = i^2(i^2-1)^2$. If $i$ is odd, then $(i^2-1)$ is divisible by $2$ at least $3$ times, hence $(i^2-1)$ is divisible by $2$ at least $6$ times. If $i$ is even and $a_i$ non-zero, then $i$ is one of $d_1,d_2,d'_1,d'_2$. We consider each these cases in turn. For $i = d_1 = 4u_1$, we have
\[
f_4( d_1) = 2^4 u_1^2 ( d_1^2-1)^2 = 0 \; ({\rm mod} \; 2^4).
\]
Similarly, $f_4( d'_1) = 0 \; ({\rm mod} \; 2^4)$. For $i = d_2 = 2u_2$, we have
\[
f_4(d_2) = 4 u_2^2 ( d_2^1 - 1)^2.
\]
Note that $u_2( d_2^1-1)^2$ is the square of an odd number hence $u_2( d_2^1-1)^2 = 1 \; ({\rm mod} \; 4)$ and $f_4(d_2) = 4 \; ({\rm mod} \; 2^4)$. Similarly $f_4(d'_2) = 4 \; ({\rm mod} \; 2^4)$. Hence
\[
\sigma_{n+1}(\underline{d}) - \sigma_{n+1}(\underline{d}') = f_4(d_2) - f_4(d'_2) = 4-4 = 0 \; ({\rm mod} \; 2^4)
\]
so that $2$ divides $\sigma_{n+1}(\underline{d}) - \sigma_{n+1}(\underline{d}')$ at least $4$ times, as required.

In the case $n=9$, we have $\nu_2( (n+1)!) = \nu_2(10!) = 8$. Consider $f_5(i) = i^4(i^2-1)^3$. If $i$ is odd, then $(i^2-1)$ is divisible by $2$ at least $3$ times, hence $(i^2-1)^3$ is divisible by $2$ at least $9$ times. If $i$ is even and $a_i$ non-zero, then $i$ is one of $d_1,d_2,d'_1,d'_2$. We consider each these cases in turn. For $i = d_1 = 4u_1$, we have that $f_5(d_1)$ is divisible by $2^8$ and similarly for $f_5(d'_1)$. For $i = d_2 = 2u_2$, we have
\[
f_5(d_2) = 2^4 u_2^4 ( d_2^2-1)^3.
\]
Now since $u_2$ is odd, it follows that $u_4^4 = 1 \; ({\rm mod} \; 16)$. We also have that $d_2^2-1 = 4u_2^2-1$. Then since $u_2$ is odd, $u_2^2 = 1 \; ({\rm mod} \; 4)$ and hence $4u_2^2-1 = 4-1 = 3 \; ({\rm mod} \; 16)$. So $(d_2^2-1)^3 = 3^3 = 11 \; ({\rm mod} \; 16)$ and so $f_5(d_2) = 2^4 \cdot 11 \; ({\rm mod} \; 2^8)$. The same argument gives $f_5(d'_2) = 2^4 \cdot 11 \; ({\rm mod} \; 2^8)$. Hence
\[
\sigma_{n+1}(\underline{d}) - \sigma_{n+1}(\underline{d}') = f_5(d_2) - f_5(d'_2) = 2^4 \cdot 11 - 2^4 \cdot 11 = 0 \; ({\rm mod} \; 2^8)
\]
so that $2$ divides $\sigma_{n+1}(\underline{d}) - \sigma_{n+1}(\underline{d}')$ at least $8$ times, as required.

In the case $n=13$, we have $\nu_2( (n+1)!) = \nu_2(14!) = 11$. Consider $f_6(i) = i^2(i^2-2)^2(i^2-1)^4$. If $i$ is odd, then $(i^2-1)$ is divisible by $2$ at least $3$ times, hence $(i^2-1)^4$ is divisible by $2$ at least $12$ times. If $i$ is even and $a_i$ non-zero, then $i$ is one of $d_1,d_2,d'_1,d'_2$. We consider each these cases in turn. For $i = d_1 = 4u_1$, we have 
\[
f_6(d_1) = 2^4 u_1^2 ( (4u_1)^2-4)^2( (4u_1)^2-1)^4 = 2^8 u_1^2 ( (2u_1)^2-1)^2 ( (4u_1)^2-1)^4.
\]
Note that $u_1^2 ( (2u_1)^2-1)^2 ( (4u_1)^2-1)^4$ is the square of an odd number, hence $u_1^2 ( (2u_1)^2-1)^2 ( (4u_1)^2-1)^4 = 1 \; ({\rm mod} \; 8)$ and so $f_6(d_1) = 2^8 \; ({\rm mod} \; 2^{11})$. Similarly $f_6(d'_1) = 2^8 \; ({\rm mod} \; 2^{11})$.

For $i = d_2 = 2u_2$, we get
\[
f_6(d_2) = 2^2 u_2^2 (  4( u_2^2-1) )^2 ( (2u_2)^2-1)^4 = 2^6 u_2^2 ( u_2^2-1)^2( (2u_2)^2-1)^4.
\]
Note that since $u_2$ is odd, $u_2^2-1$ is divisible by $8$ and hence $(u_2^2-1)^2$ is divisible by $8^2 = 2^6$. Hence $2$ divides $f_6(d_2)$ at least $12$ times. Similarly, $2$ divides $f_6(d'_2)$ at least $12$ times. Hence
\[
\sigma_{n+1}(\underline{d}) - \sigma_{n+1}(\underline{d}') = f_5(d_1) - f_5(d'_1) = 2^8 - 2^8 = 0 \; ({\rm mod} \; 2^{11})
\]
so that $2$ divides $\sigma_{n+1}(\underline{d}) - \sigma_{n+1}(\underline{d}')$ at least $11$ times, as required.

\end{proof}

\section{Formulas for the alpha invariant}\label{sec:formulas}

In this section we examine the alpha invariant from several points of views and derive corresponding formulas for it. These points of views can be categorised as topological, geometric and algebraic.

First we consider the topological point of view. Let $n\ge 1$ be an integer with $n = 1 \; ({\rm mod} \; 4)$ and let $d_1, \dots , d_k$ be positive integers such that an even number of the $d_i$ are even. Assume as in the previous sections that $k$ is a multiple of $4$. Then the complete intersection $X_n(\underline{d})$ has a unique spin structure and from Proposition \ref{prop:alphan} we have
\[
\alpha_n(\underline{d}) = \pi_*^{n+k}(\beta),
\]
where $\beta \in K^0(\mathbb{CP}^{n+k})$ is given by Equation (\ref{equ:beta}). Let $N = 1 \; ({\rm mod} \; 4)$, $N \ge 1$. Identifying harmonic spinors with Dolbeault cohomology, one finds easily that the index map $\pi_*^{N} : K^0( \mathbb{CP}^{N} ) \to \mathbb{Z}$ is given by:
\begin{equation}\label{equ:pistar}
\pi_*^N( \xi^m ) = \binom{ \frac{N-1}{2} + m }{N}.
\end{equation}

From this we obtain:
\begin{theorem}\label{thm:alpha3}
Let $n = 1 \; ({\rm mod} \; 4)$ and $\underline{d}$ be given. Suppose that an even number of the $d_i$ are even. Then
\begin{equation}\label{equ:alpha3}
\alpha_n(\underline{d}) = \sum_{\epsilon_2 , \dots , \epsilon_k = \pm 1} \binom{\frac{1}{2}(n+k-1+\underline{\epsilon} \cdot \underline{d}) }{n+k} \; ({\rm mod} \; 2)
\end{equation}
where $\underline{\epsilon} = (1 , \epsilon_2 , \dots , \epsilon_k)$ and $\underline{\epsilon} \cdot \underline{d} = d_1 + \epsilon_2 d_2 + \cdots + \epsilon_k d_k$.
\end{theorem}
\begin{proof}
If $k$ is a multiple of $4$, then this follows immediately from Proposition \ref{prop:alphan} and Equation (\ref{equ:pistar}). Below we will give another proof which does not require $k$ to be a multiple of $4$. Alternatively, one can verify that the right hand side of (\ref{equ:alpha3}) is invariant mod $2$ under $k \mapsto k+1$, $(d_1 , \dots ,d_k) \mapsto (1 , d_1 , \dots , d_k )$ by Pascal's formula.
\end{proof}

Equation (\ref{equ:alpha3}) can be written a little more symmetrically as follows:

\begin{corollary}
Let $n = 1 \; ({\rm mod} \; 4)$ and $\underline{d}$ be given. Suppose that an even number of the $d_i$ are even. Then
\[
\alpha_n(\underline{d}) = \sum_{\epsilon_1 , \epsilon_2 , \dots , \epsilon_k = \pm 1} \binom{\frac{1}{2}(n+k+1+\underline{\epsilon} \cdot \underline{d}) }{n+k+1}
\]
where $\underline{\epsilon} = (\epsilon_1 , \epsilon_2 , \dots , \epsilon_k)$ and $\underline{\epsilon} \cdot \underline{d} = \epsilon_1 d_1 + \epsilon_2 d_2 + \cdots + \epsilon_k d_k$.
\end{corollary}
\begin{proof}
Suppose $\underline{d} = (d_1 , \dots , d_k)$. Then $\alpha_n( \underline{d}) = \alpha_n( \underline{d}')$, where $\underline{d}' = (1 , d_1 , \dots , d_k)$. Applying Theorem \ref{thm:alpha3} to $\underline{d}'$ gives the result.
\end{proof}

Next, we consider $\alpha_n(\underline{d})$ from a geometric point of view. Recall that $X_n(\underline{d}) \subseteq \mathbb{CP}^{n+k}$ is a complete intersection and hence inherits a K\"ahler structure from the projective space $\mathbb{CP}^{n+k}$. Let $H^\pm( X_n(\underline{d}))$ denote the space of harmonic spinors on $X_n(\underline{d})$ with respect to the induced K\"ahler metric. Then by Hodge theory we have:
\[
H^\pm(X_n(\underline{d})) \cong H^{ev/odd}( X_n(\underline{d}) ; \mathcal{O}(m) ).
\]
Recall that the pushforward map is realised analytically by taking the index of the Dirac operator. In particular in dimensions $2 \; ({\rm mod} \; 8)$, the push-forward is the mod $2$ index:
\[
\alpha( X_n(\underline{d})) = \dim ( H^{\pm}( X_n(\underline{d}) ) ) \; ({\rm mod} \; 2).
\]

\begin{lemma}\label{lem:concentrate}
Let $X_n(\underline{d})$ be any complete intersection (where $n$ can be any positive integer) and let $u$ be any integer. Then $H^*( X_n(\underline{d}) ; \mathcal{O}(u))$ is non-zero only in degrees $0$ and $n$.
\end{lemma}
\begin{proof}
We prove this by induction on $k$. If $k=0$ then $X_n(\underline{d}) = \mathbb{CP}^n$ and the result is obviously true. Now let $\underline{d} = (d_1, \dots , d_k)$ and consider $\underline{d}' = (d_1 , \dots , d_k , d_{k+1})$. We view $X' = X_n(\underline{d}')$ as a hypersurface in $X = X_{n+1}(\underline{d})$. Then we have the short exact sequence of sheaves:
\[
0 \to \mathcal{O}_{X}( u') \to \mathcal{O}_{X}( u) \to \mathcal{O}_{X'}( u) \to 0
\]
where $u' = u-d_{k+1}$. Taking the associated long exact sequence gives the result.
\end{proof}

\begin{corollary}\label{cor:h0}
Let $X_n(\underline{d})$ be a complete intersection of dimension $n = 1 \; ({\rm mod} \; 4)$ and where an even number of $d_i$ are even. Then
\[
\alpha(X_n(\underline{d})) = dim( H^0( X_n(\underline{d}) ; \mathcal{O}(m) ) ) ) \; ({\rm mod} \; 2),
\]
where $2m = -n-k-1 + \sum_i d_i$.
\end{corollary}
\begin{proof}
By Lemma \ref{lem:concentrate}, the non-zero cohomology of $\mathcal{O}(m)$ is concentrated in degrees $0$ and $n$. Then since $n$ is odd, we have $H^+(X_n(\underline{d})) \cong H^0(X_n(\underline{d}) ; \mathcal{O}(m))$ and the result follows.
\end{proof}

Recall that the Hilbert series of a projective variety $X \subseteq \mathbb{CP}^N$ is the formal power series with integer coefficients given by:
\[
H_X(t) = \sum_{j \ge 0} dim( H^0(X ; \mathcal{O}(j) ) ) t^j.
\]
It is well known that for a complete intersection $X_n(\underline{d})$, the Hilbert series is given by:
\[
H_{X_n(\underline{d})}(t) = \left( \frac{1}{1-t} \right)^{n+k+1} (1-t^{d_1})(1-t^{d_2}) \cdots (1 - t^{d_k})
\]
where $1/(1-t)$ is to be understood as the formal power series $1+t+t^2 + \cdots $. Combined with Corollary \ref{cor:h0} we have shown:
\begin{theorem}\label{thm:alpha}
Let $n = 1 \; ({\rm mod} \; 4)$ and suppose that $\underline{d} = (d_1, \dots , d_k)$ where an even number of the $d_i$ are even. Then $\alpha_n(\underline{d})$ is the mod $2$ reduction of the coefficient of $t^m$ in 
\[
\left( \frac{1}{1-t} \right)^{n+k+1} (1-t^{d_1})(1-t^{d_2}) \cdots (1 - t^{d_k}),
\]
where $2m = -n-k-1+\sum_i d_i$.
\end{theorem}

\begin{corollary}\label{cor:alpha4}
We have that $\alpha_n(\underline{d})$ is the mod $2$ reduction of the coefficient of $t^0$ in
\[
\left( \frac{t}{1-t^2} \right)^{n+k+1} (t^{d_1}+t^{-d_1}) \cdots (t^{d_k} + t^{-d_k}).
\]
\end{corollary}
\begin{proof}
By Theorem \ref{thm:alpha3}, $\alpha_n(\underline{d})$ is the mod $2$ reduction of the coefficient of $t^{2m}$ in
\[
\left( \frac{1}{1-t^2} \right)^{n+k+1} (1-t^{2d_1})(1-t^{2d_2}) \cdots (1 - t^{2d_k}).
\]
Note that since we are working mod $2$ we can replace the factors $(1-t^{2d_j})$ by $(1+t^{2d_j})$. Recall that $2m = -n-k-1+\sum_i d_i$. Multiplying each factor $(1+t^{2d_j})$ by $t^{-d_j}$ and multiplying $(1-t^2)^{-n+k+1}$ by $t^{n+k+1}$, we see that $\alpha_n(\underline{d})$ is the mod $2$ reduction of the $t^0$-coefficient in
\[
\left( \frac{t}{1-t^2} \right)^{n+k+1} (t^{d_1}+t^{-d_1}) \cdots (t^{d_k} + t^{-d_k}).
\]
\end{proof}

Now we will study $\alpha_n(\underline{d})$ from an algebraic viewpoint. Theorem \ref{thm:alpha} suggests a purely algebraic definition of $\alpha_n(\underline{d})$ that applies even when $n \neq 1 \; ({\rm mod} \; 4)$ or when some of the $d_i$ are negative. Adopting this point of view will allow us to derive some further algebraic properties of the invariant.

Let $n \ge -1$ be an integer and let $d_1, \dots , d_k$ be integers, possibly zero or negative. We write $\underline{d} = (d_1, \dots , d_k)$. Consider the formal Laurent series
\begin{equation*}
\begin{aligned}
K_n(\underline{d}) &= \left( \frac{1}{1-t} \right)^{n+k+1} (1+t^{d_1}) \cdots (1 + t^{d_k}) \\
&= \sum_{j \ge 0} \binom{n+k+j}{j} t^j (1+t^{d_1}) \cdots (1 + t^{d_k}).
\end{aligned}
\end{equation*}

Note that the coefficients of $K_n(\underline{d})$ are integers.

\begin{definition}
Let $n\ge -1$ and $\underline{d}$ be given. We define the abstract alpha invariant $\alpha_n(\underline{d}) \in \mathbb{Z}_2$ to be the mod $2$ reduction of the coefficient of $t^m$ in $K_n(\underline{d})$, where $2m = -n-k-1+\sum_i d_i$. If $-n-k-1+\sum_i d_i$ is odd then we set $\alpha_n(\underline{d}) = 0$.
\end{definition}

If $\underline{d} = (d_1, \dots , d_k)$, we will also write $\alpha_n(\underline{d})$ as $\alpha_n(d_1 , \dots , d_k)$. If $n \ge 1$, $n = 1 \; ({\rm mod} \; 4)$, $d_i \ge 1$ for each $i$ and an even number of $d_i$ are even, then $X  = X_n(\underline{d})$ is a spin complete intersection and $\alpha_n(d_1 , \dots , d_k)$ is the alpha invariant of $X$. This follows since the coefficients of $K_n(\underline{d})$ agree mod $2$ with the coefficients of the Hilbert series $H_X(t)$, hence $\alpha_n( \underline{d}) = \alpha(X)$ by Theorem \ref{thm:alpha}.

\begin{lemma}\label{lem:1}
We have:
\begin{itemize}
\item[(i)]{If $d_i = 0$ for some $i$, then $\alpha_n(\underline{d}) = 0$.}
\item[(ii)]{$\alpha_n( d_1, \dots , d_i , \dots , d_k) = \alpha_n( d_1, \dots , -d_i , \dots , d_k)$.}
\end{itemize}
\end{lemma}
\begin{proof}
If $d_i=0$ for some $i$ then $(1+t^d_i) = 1+1 = 2$ is a factor of $K_n(\underline{d})$. Then all the coefficients of $K_n(\underline{d})$ are even and hence $\alpha_n( \underline{d}) = 0$.

Next, suppose we replace $d_i$ by $-d_i$. Then $m \mapsto m - d_i$. On the other hand $(1+t^{d_i}) = t^{d_i}( 1+t^{-d_i})$ and it follows that the degree $m-d_i$ coefficient of $H_n( d_1, \dots , -d_i , \dots , d_k)$ is the degree $m$ coefficient of $K_n( d_1 , \dots , d_i , \dots , d_k)$. So $\alpha_n( d_1, \dots , d_i , \dots , d_k) = \alpha_n( d_1, \dots , -d_i , \dots , d_k)$.
\end{proof}

\begin{lemma}\label{lem:2}
Suppose that $k=1$ and $\underline{d} = (d)$. Then
\[
\alpha_n( d ) = \begin{cases} \binom{ \frac{1}{2}(n+d) }{n+1 } & \text{if } n+d \text{ is even}, \\ 0 & \text{otherwise}. \end{cases}
\]
\end{lemma}
\begin{proof}
If $n+d$ is odd then $\alpha_n(d)=0$ by definition. Now suppose that $n+d$ is even. First consider the case $d>0$. Then $2m = -n-2+d < d$ and hence $m < d/2 < d$. So the coefficient of $t^m$ in
\[
\left(\frac{1}{1-t} \right)^{n+2} (1+t^d) = \sum_{j \ge 0 } \binom{n+1+j}{n+1} (t^j + t^{j+d})
\]
is $\binom{n+1+m}{n+1}$. Working mod $2$ and using $2m = -n-2+d$, we find:
\[
\binom{n+1+m}{n+1} = \binom{\frac{1}{2}(n+d)}{n+1}.
\]
Now suppose $d < 0$. Then
\[
\alpha_n(d) = \alpha_n( -d ) = \binom{\frac{1}{2}(n-d)}{n+1}.
\]
But using the identity $\binom{a}{b} = (-1)^b \binom{b-a-1}{b} = \binom{b-a-1}{b} \; ({\rm mod} \; 2)$, we get that
\[
\binom{\frac{1}{2}(n-d)}{n+1} = \binom{\frac{1}{2}(n+d)}{n+1} \; ({\rm mod} \; 2).
\]
Lastly, suppose that $d=0$ and $n+d = n$ is even (then $n \ge 0$). Then by Lemma \ref{lem:1} we have $\alpha_n(0) = 0$. On the other hand, we find
\[
\binom{\frac{1}{2}(n+d)}{n+1} = \binom{\frac{1}{2}n}{n+1} = 0
\]
since $n \ge 0$.
\end{proof}

\begin{lemma}\label{lem:3}
Suppose that $n \ge -1$ and $k \ge 2$. We have the following identity:
\[
\alpha_n( d_1, d_2, \dots , d_k ) = \alpha_{n+1}( d_1 + d_2 , d_3 ,  \dots , d_k ) + \alpha_{n+1}( d_1-d_2 , d_3 , \dots , d_k ).
\]
\end{lemma}
\begin{proof}
The idea behind this identity is the simple observation that
\[
(1+t^{d_1})(1+t^{d_2}) = (1+t^{d_1+d_2}) + t^{d_2}(1+t^{d_1-d_2})
\]
and therefore
\begin{equation}\label{equ:h}
K_n( \underline{d}) = K_{n+1}( \underline{d}_+ ) +  t^{d_2} K_{n+1}(\underline{d}_- )
\end{equation}
where
\[
\underline{d} = (d_1, \dots , d_k), \quad \underline{d}_+ = (d_1+d_2 , d_3 , \dots , d_k), \quad \underline{d}_- = (d_1-d_2 , d_3 , \dots , d_k).
\]
Let $m_n(\underline{d})$ be defined by $2m_n(\underline{d}) = -n-k-1 + \sum_i d_i$. Then it is straightforward to check that
\[
m_{n+1}( \underline{d}_+ ) = m_n(\underline{d}), \quad m_{n+1}( \underline{d}_-) = m_n(\underline{d}) - d_2.
\]
From this and Equation (\ref{equ:h}) it follows immediately that
\[
\alpha_n( d_1, d_2, \dots , d_k ) = \alpha_{n+1}( d_1 + d_2 , d_3 ,  \dots , d_k ) + \alpha_{n+1}( d_1-d_2 , d_3 , \dots , d_k ).
\]
\end{proof}

\begin{remark}
Using Lemmas \ref{lem:2} and \ref{lem:3} and induction on $k$, we obtain another proof of Theorem \ref{thm:alpha}
\end{remark}

\begin{lemma}\label{lem:n=-1}
For all $(d_1, \dots , d_k)$, we have
\[
\alpha_{-1}(d_1 , \dots , d_k) = d_1 \cdots d_k.
\]
\end{lemma}
\begin{proof}
Let $\underline{d} = (d_1, \dots , d_k)$. By Lemma \ref{lem:1}, it suffices to assume $d_i > 0$ for all $i$. Arguing as in the proof of Corollary \ref{cor:alpha4}, we see that $\alpha_{-1}(\underline{d})$ is the $t^0$-coefficient of
\[
\left( \frac{t}{1-t^2} \right)^{k} (t^{d_1}-t^{-d_1}) \cdots (t^{d_k} - t^{-d_k})).
\]
Equivalently, $\alpha_{-1}(\underline{d})$ is the $t^0$-coefficient of the Laurent polynomial
\begin{equation}\label{equ:prod1}
\left( \frac{ t^{d_1} - t^{-d_1}}{t - t^{-1}} \right) \cdots \left( \frac{ t^{d_k} - t^{-d_k}}{t - t^{-1}} \right).
\end{equation}
Noting that
\[
\frac{ t^a - t^{-a}}{t - t^{-1}} = t^{-a+1} + t^{-a+3} + \cdots + t^{a-1},
\]
it follows that if $a$ is odd then
\[
\frac{ t^a - t^{-a}}{t - t^{-1}} = 1 + h_a(t) + h_a(t^{-1})
\]
with $h_a(t) = t^2 + t^4 + \cdots + t^{a-1}$. Similarly, if $a$ is even, then
\[
\frac{ t^a - t^{-a}}{t - t^{-1}} = h_a(t) + h_a(t^{-1})
\]
with $h_a(t) = t + t^3 + \cdots + t^{a-1}$. Taking coefficients mod $2$, either case can be written as
\[
\frac{ t^a - t^{-a}}{t - t^{-1}} = a + h_a(t) + h_a(t^{-1}).
\]
Substituting $d_i + h_{d_i}(t) + h_{d_i}(t^{-1})$ for each factor in (\ref{equ:prod1}), we see by induction on $k$ that this product equals
\[
d_1 \cdots d_k + r(t) + r(t^{-1}) \; ({\rm mod} \; 2)
\]
for some Laurent polynomial $r(t)$ with integer coefficients. Now since $r(t)$ and $r(t^{-1})$ have the same coefficient of $t^0$ it follows that the coefficient of $t^0$ in (\ref{equ:prod1}) taken mod $2$ is $d_1 \cdots d_k$. Hence $\alpha_{-1}(\underline{d}) = d_1 \cdots d_k$.
\end{proof}

We now extend the definition of $\alpha_n(\underline{d})$ to all integer values of $n$ by declaring that $\alpha_n(\underline{d}) = 0$ whenever $n < -1$. Note that the recursive equation
\[
\alpha_n( d_1, d_2, \dots , d_k ) = \alpha_{n+1}( d_1 + d_2 , d_3 ,  \dots , d_k ) + \alpha_{n+1}( d_1-d_2 , d_3 , \dots , d_k )
\]
is still valid for all $n$. The only case that needs checking is $n=-2$, which is equivalent to
\[
\alpha_{-1}( d_1 + d_2 , d_3 ,  \dots , d_k ) = \alpha_{-1}( d_1-d_2 , d_3 , \dots , d_k ),
\]
which is easily verified by Lemma \ref{lem:n=-1}.

For a given $\underline{d} = (d_1, \dots , d_k)$, we may view $\alpha_n(\underline{d})$ as a function 
\[
\mathbb{Z} \to \mathbb{Z}_2 \quad n \mapsto \alpha_n(\underline{d}).
\]
Let $T$ be the shift operator which acts on functions $f : \mathbb{Z} \to \mathbb{Z}_2$ by $(Tf)(n) = f(n-1)$. Then $T$ generates an action of $\mathbb{Z}_2[T]$ on such functions.

\begin{proposition}\label{prop:fr}
For each $r \ge 0$, there exists a polynomial $f_r(T) \in \mathbb{Z}_2[T]$ such that
\[
\alpha_n( r , \underline{d} ) = f_r(T) \alpha_n( \underline{d})
\]
for all $\underline{d}$. Moreover we have that $f_r$ satisfies the following recursion relation
\[
f_0 = 0, \quad f_1 = 1, \quad f_r = T f_{r-1} + f_{r-2}, \; r \ge 2.
\]
\end{proposition}
\begin{proof}
We prove this by induction on $r$. From Lemma \ref{lem:1}, we have $\alpha_n( 0 , \underline{d}) = 0$. From the definition of $\alpha_n(\underline{d})$ we also clearly have $\alpha_n( 1 , \underline{d}) = \alpha_n( \underline{d})$. This proves the $r=0,1$ cases. Now let $r \ge 2$ and assume the result holds up to $r-1$. By Lemma \ref{lem:3} we have
\[
\alpha_{n-1}( r-1 , \underline{d} ) = \alpha_{n-1}( 1 , r-1 , \underline{d}) = \alpha_n( r , \underline{d}) + \alpha_n( r-2 , \underline{d}).
\]
So by re-arranging and using the inductive hypothesis we get
\[
\alpha_n(r , \underline{d}) = \alpha_{n-1}( r-1 , \underline{d}) + \alpha_n( r-2 , \underline{d}) = (T f_{r-1}  + f_{r-2}) \alpha_n(\underline{d}),
\]
which proves the result.
\end{proof}

\begin{proposition}\label{prop:fr2}
Let $r \ge 1$. We have that
\[
f_r(T) = \sum_{j \ge 0} \binom{r+j}{2j+1} T^j.
\]
\end{proposition}
\begin{proof}
First note that by induction on $r$ one finds that $f_r(T)$ only involves odd powers of $T$ if $r$ is even and only even powers of $T$ if $r$ is odd. For $d \ge 0$ and $0 \le u \le d/2$, let $c^u_d$ denote the $T^d$-coefficient of $f_{2u-d+1}(T)$. Then the recursion relation $f_r = Tf_{r-1} + f_{r-2}$ gives $c^u_d = c^{u-1}_{d-1} + c^{u-1}_d$ (or $c^u_0 = c^{u-1}_0$ when $d=0$). Moreover $f_1(T) = 1$ gives $c^0_0 = 1$. It follows that $c^u_d$ satisfies the same recursion relation (Pascal's formula) and initial conditions as the binomial coefficients $\binom{u}{d}$. Hence $c^u_d = \binom{u}{d}$. Now setting $r = 2u-d+1$, we see that the $T^d$-coefficient of $f_r(T)$ is $\binom{u}{d} = \binom{ (r+d-1)/2 }{d}$ if $r+d$ is odd and zero otherwise. We claim that this coincides mod $2$ with $\binom{r+d}{2d+1}$. In fact, if $r+d$ is odd, then
\[
\binom{r+d}{2d+1} = \binom{ r+d-1}{2d} \binom{1}{1} = \binom{(r+d-1)/2 }{d} \; ({\rm mod} \; 2)
\]
and if $r+d$ is even, then
\[
\binom{r+d}{2d+1} = \binom{r+d}{2d} \binom{0}{1} = 0 \; ({\rm mod} \; 2)
\]
which proves the claim.
\end{proof}

\begin{lemma}\label{lem:empty}
Let $\underline{d} = \emptyset$. Then
\[
\alpha_n( \emptyset ) = \begin{cases} 1 & \text{if } n = -1, \\ 0 & \text{otherwise}. \end{cases}
\]
\end{lemma}
\begin{proof}
Clearly $\alpha_n(\emptyset)$ is zero unless $n$ is odd and $n \ge -1$. On the other hand if $n > 0$ is odd and $2m = -n-1$, then $m < 0$ and it follows again that $\alpha_n(\emptyset)$ is zero. This leaves only the case $n=-1$ for which $\alpha_n(\emptyset) = 1$ by Lemma \ref{lem:n=-1}.
\end{proof}

\begin{theorem}\label{thm:alpha2}
Let $\underline{d} = (d_1, \dots , d_k)$, where $d_i \ge 1$. We have that $\alpha_n(\underline{d})$ is the coefficient of $T^{n+1}$ in $f_{d_1}(T)f_{d_2}(T) \cdots f_{d_k}(T)$.
\end{theorem}
\begin{proof}
By Proposition \ref{prop:fr} we have that
\[
\alpha_n(\underline{d}) = f_{d_1}(T)f_{d_2}(T) \cdots f_{d_k}(T) \alpha_n( \emptyset).
\]
The result now follows easily from Lemma \ref{lem:empty}
\end{proof}

\begin{corollary}\label{cor:alpha2}
Let $\underline{d} = (d_1, \dots , d_k)$. Then
\[
\alpha_n(\underline{d}) = \sum_{j_1 + \cdots + j_k = n+1} \binom{d_1+j_1}{2j_1+1} \cdots \binom{d_k+j_k}{2j_k+1}.
\]
\end{corollary}
\begin{proof}
In the case that $d_i \ge 0$, this follows easily from Theorem \ref{thm:alpha2} and Proposition \ref{prop:fr}. If $d_i = 0$ for some $i$, then $\alpha_n(\underline{d}) = 0$. On the other hand, $\binom{j_i}{2j_i+1} = 0$ for all $j_i \ge 0$, hence the expression on the right is also zero in this case. Lastly, if $d_i < 0$ for some $i$, then we note that
\[
\binom{d_i+j_i}{2j_i+1} = \binom{ (2j_1+1)-d_i-1}{2j_i+1} = \binom{ -d_i + j_i}{2j_i +1} \; ({\rm mod} \; 2).
\]
Hence the expression on the right is invariant under $d_i \mapsto -d_i$ for each $i$. Replacing each $d_i$ with $d_i < 0$ by $-d_i$, we see that the expression on the right agrees with $\alpha_n( \underline{d})$.
\end{proof}

\begin{corollary}
Let $\underline{d} = (d_1, \dots , d_k)$. Then
\[
\alpha_n(\underline{d}) = \sum_{j_1 + \cdots + j_k = n+1 \atop j_i+d_i \text{ odd}} \binom{\frac{1}{2}(d_1+j_1-1)}{j_1} \cdots \binom{\frac{1}{2}(d_k+j_k-1)}{j_k}.
\]
\end{corollary}
\begin{proof}
Follows easily from Corollary \ref{cor:alpha2} and the fact that $\binom{j+d}{2j+1} = 0 \; ({\rm mod} \; 2)$ if $j+d$ is even and equals $\binom{\frac{1}{2}(j+d-1)}{j} \; ({\rm mod} \; 2)$ if $d+j$ is odd.
\end{proof}

\section{The case $n=1$}\label{sec:n=1}

In this section we consider the alpha invariant for $n=1$. Any one-dimensional complete intersection $X  = X_1(\underline{d})$ is a compact Riemann surface and hence is spin. The Euler class is the first Chern class, which by Equation (\ref{equ:chern}) is given by
\[
c_1(TX) = (2+k - \sum_i d_i)x = (2+k - \sigma_1)x.
\]
But note that $x$ is $d_{tot}$ times a generator of $H^2(X ; \mathbb{Z})$, so
\[
\chi(X) = 2-2g = (2+k-\sigma_1)d_{tot}
\]
where $g$ is the genus of $X$. A genus $g$ Riemann surface has $2^{2g}$ spin structures. Note that $g=0$ is only possible if $\underline{d} = (1,\dots , 1)$ or $(2,1, \dots , 1)$. In all other cases $g>0$ and $X$ has multiple spin structures. Nevertheless, if an even number of the $d_i$ are even then $X$ has a distinguished spin structure determined by the unique square root of the canonical bundle which extends to a line bundle on $\mathbb{CP}^{k+1}$. Let $\alpha_1(\underline{d})$ denote the alpha invariant of this spin structure.

\begin{theorem}\label{thm:alpha1}
Let $\underline{d} = (d_1, \dots , d_k)$ and suppose that an even number of $d_i$ are even (so $\nu_2(d_{tot}) \neq 1$).
\begin{itemize}
\item[(i)]{If $\nu_2(d_{tot}) \ge 3$, then $\alpha_1( \underline{d}) = 0$.}
\item[(ii)]{If $\nu_2(d_{tot}) = 2$, then $\alpha_1( \underline{d}) = 1$.}
\item[(iii)]{If $\nu_2(d_{tot}) = 0$, then 
\[
\alpha_1(\underline{d}) = \begin{cases} 0 & \text{if } d_{tot} = \pm 1 \; ({\rm mod} \; 8) \\ 1 & \text{if } d_{tot} = \pm 3 \; ({\rm mod} \; 8). \end{cases}
\]

}
\end{itemize}
In particular, it follows that $\alpha_1(\underline{d})$ depends only on $d_{tot}$.
\end{theorem}

\begin{remark}
It is interesting to compare the formula for $\alpha_1(\underline{d})$ given by Theorem \ref{thm:alpha1} with that of the Kervaire invariant of a complete intersection $X_n(\underline{d})$ when this is defined \cite{bro,woo}.
\end{remark}

\begin{proof}
By Theorem \ref{thm:alpha2}, $\alpha_1(\underline{d})$ is the $T^2$-coefficient of $f_{d_1}(T) \cdots f_{d_k}(T)$. From Proposition \ref{prop:fr2}, one finds that $f_r(T) = (r/2)T + O(T^2)$ for even $r$ and $f_r(T) = 1 + \binom{ (r+1)/2}{2}T^2 + O(T^3)$ for odd $r$. It follows easily that $\alpha_1(\underline{d})=0$ for $\nu_2(d_{tot}) \ge 3$. If $\nu_2(d_{tot}) = 2$, then (since the number of even $d_i$ is even) exactly two of the $d_i$ are even and equal $2 \; ({\rm mod} \; 4)$ and all other $d_i$ are odd. It follows easily that in this case $\alpha_1(\underline{d})=1$.

It remains to consider the case where all $d_i$ are odd. Then from $f_{d_i}(T) = 1 + \binom{ (d_i+1)/2}{2}T^2 + O(T^3)$, we get
\[
\alpha_1(\underline{d}) = \sum_i \binom{ (d_i+1)/2}{2}.
\]
Note that for $d_i$ odd, $\binom{ (d_i+1)/2}{2}$ is odd if and only if $d_i = \pm 3 \; ({\rm mod} \; 8)$. On the other hand one finds that $d_i^2-1 = 8 \; ({\rm mod} \; 16)$ for $d_i = \pm 3 \; ({\rm mod} \; 8)$ and $d_i^2 - 1 = 0 \; ({\rm mod} \; 16)$ for $d_i = \pm 1 \; ({\rm mod} \; 8)$. So
\[
\alpha_1(\underline{d}) = \sum_i \frac{d_i^2-1}{8} = \frac{\sigma_2 - k}{8}.
\]
Moreover if $d_i$ is odd, then by considering the prime factorisation of $d_i$, one sees that 
\[
\frac{d^2_i-1}{8} = \sum_{p = \pm 3 \; ({\rm mod} \; 8)} \nu_p( d_i) \; ({\rm mod} \; 2)
\]
where the sum is over all primes congruent to $\pm 3$ mod $8$. Summing over $i$, we get
\[
\alpha_1(\underline{d}) = \sum_{p = \pm 3 \; ({\rm mod} \; 8)} \nu_p( d_{tot}) = \begin{cases} 0 & \text{if } d_{tot} = \pm 1 \; ({\rm mod} \; 8) \\ 1 & \text{if } d_{tot} = \pm 3 \; ({\rm mod} \; 8). \end{cases}
\]

\end{proof}

\section{$\hat{A}$-genus of even dimensional spin complete intersections}\label{sec:ahat}

The techniques used in this paper to compute the $\alpha$ invariant for complete intersections in dimensions $ 1 \; ({\rm mod} \; 4)$ can also be used to compute the $\hat{A}$-genus of even dimensional spin complete intersections. In this section we briefly summarise the results.

Let $n \ge 2$ be even and let $\underline{d} = (d_1, \dots , d_k)$ be such that an odd number of $d_i$ are even. Then $X = X_n(\underline{d})$ has a unique spin structure and we let $\hat{A}[X]$ denote the $\hat{A}$-genus of $X_n(\underline{d})$. Then by definition of the $\hat{A}$-genus and Equation (\ref{equ:pont}), we immediately have
\[
\hat{A}[X] = \int_X \left( \frac{x}{e^{x/2} - e^{-x/2}} \right)^{n+k+1} \prod_{j=1}^k \left( \frac{e^{d_ix/2} - e^{-d_ix/2}}{d_i x} \right).
\]
Noting that $\int_X x^n = d_{tot}$, it follows that $\hat{A}[X]$ is the coefficient of $x^n$ in
\[
\left( \frac{x}{e^{x/2} - e^{-x/2}} \right)^{n+1} \prod_{j=1}^k \left( \frac{e^{d_ix/2} - e^{-d_ix/2}}{e^{x/2} - e^{-x/2}} \right).
\]

Now we consider two different methods for computing $\hat{A}[X]$ analogous to the topological and geometric approaches to computing $\alpha_n(\underline{d})$. First the topological method. We have that $\hat{A}[X] = \rho_*(1)$, where $\rho_* : K^0(X) \to K^0(pt)$ is the pushforward in complex $K$-theory. Assume that $k$ is odd so that $\mathbb{CP}^{n+k}$ is spin. Arguing as in the proof of Proposition \ref{prop:alphan}, but this time using complex $K$-theory gives:

\begin{proposition}\label{prop:alphan}
We have
\[
\hat{A}[X] = \pi_*^{n+k}(\beta)
\]
where
\[
\beta = \bigoplus_{\underline{\epsilon} = (\epsilon_1 , \epsilon_2 , \dots , \epsilon_k) \atop \epsilon_2 , \dots , \epsilon_k = \pm 1} sgn( \underline{\epsilon} ) \, \xi^{ (\underline{\epsilon} \cdot \underline{d}) / 2} \in K^0(\mathbb{CP}^{n+k}).
\]
\end{proposition}

Combined with Equation (\ref{equ:pistar}), we immediately obtain:
\begin{theorem}
Let $n \ge 2$ be even and $\underline{d} = (d_1, \dots , d_k)$ such that an odd number of $d_i$ are even. Then
\[
\hat{A}[X_n(\underline{d})] = \sum_{\epsilon_1 , \epsilon_2 , \dots , \epsilon_k = \pm 1} sgn( \underline{\epsilon} ) \binom{\frac{1}{2}(n+k-1+\underline{\epsilon} \cdot \underline{d}) }{n+k}
\]
where $\underline{\epsilon} = (\epsilon_1 , \epsilon_2 , \dots , \epsilon_k)$, $sgn(\underline{\epsilon}) = \epsilon_1 \cdots \epsilon_k$ and $\underline{\epsilon} \cdot \underline{d} = \epsilon_1 d_1 + \epsilon_2 d_2 + \cdots + \epsilon_k d_k$.
\end{theorem}

Note that the result holds even if $k$ is even using $X_n( d_1 , \dots , d_k) = X_n( 1 , d_1 , \dots , d_k)$ and Pascal's formula.

Next, we consider the geometric approach to computing $\hat{A}[X]$. 

\begin{theorem}
Let $n \ge 2$ be even and $\underline{d} = (d_1, \dots , d_k)$ such that an odd number of $d_i$ are even. Then $\hat{A}[X_n(\underline{d})]$ is twice the coefficient of $t^m$ in
\[
\left( \frac{1}{1-t} \right)^{n+k+1} (1-t^{d_1})(1-t^{d_2}) \cdots (1 - t^{d_k}),
\]
where $2m = -n-k-1+\sum_i d_i$.
\end{theorem}
\begin{proof}
Let $X = X_n(\underline{d})$. As in Section \ref{sec:formulas}, let Let $H^\pm( X)$ denote the space of harmonic spinors on $X$ with respect to the induced K\"ahler metric. Then by Hodge theory we have:
\[
H^\pm(X) \cong H^{ev/odd}( X ; \mathcal{O}(m) )
\]
and hence
\[
\hat{A}[X] = dim( H^+) - dim(H^-) = \sum_{j=0}^n (-1)^j dim H^j( X ; \mathcal{O}(m)).
\]
By Lemma \ref{lem:concentrate}, the cohomology groups $H^j(X ; \mathcal{O}(m))$ are non-zero only for $j=0,n$. Moreover, Serre duality gives $dim H^n(X ; \mathcal{O}(m)) = dim H^0( X ; \mathcal{O}(m) )$, since the canonical bundle is $\mathcal{O}(2m)$. Thus
\[
\hat{A}[X] = dim H^0(X ; \mathcal{O}(m)) + (-1)^n dim H^n( X ; \mathcal{O}(m)) = 2 dim H^0(X ; \mathcal{O}(m)),
\]
where we used the fact that $n$ is even. The result now follows by noting that the Hilbert series of $X$ is given by:
\[
\left( \frac{1}{1-t} \right)^{n+k+1} (1-t^{d_1})(1-t^{d_2}) \cdots (1 - t^{d_k}).
\]
\end{proof}


\bibliographystyle{amsplain}

\end{document}